\documentclass[11pt,reqno]{amsart}
\usepackage{amsmath}
\usepackage{amsfonts}
\usepackage{amssymb,latexsym}
\usepackage{cite}
\usepackage[height=190mm,width=130mm]{geometry}

\theoremstyle{plain}
\newtheorem{theorem}{Theorem}

\newtheorem{corollary}{Corollary}

\theoremstyle{definition}

\theoremstyle{remark}
\newtheorem{remark}{Remark}

\numberwithin{equation}{section}
\input{tcilatex}

\begin{document}
\title[On Horadam quaternions by using matrix method]{On Horadam quaternions
by using matrix method}
\author{Elif TAN}
\address{Department of Mathematics, Ankara University, Science Faculty,
06100 Tandogan Ankara, Turkey.}
\email{etan@ankara.edu.tr}
\author{Ho-Hon Leung}
\address{Department of Mathematical Sciences, UAEU, Al-Ain, United Arab
Emirates}
\email{hohon.leung@uaeu.ac.ae}

\begin{abstract}
In this paper, we give several matrix representations for the Horadam
quaternions. We derive several identities related to these quaternions by
using the matrix method. Since quaternion multiplication is not commutative,
some of our results are non-commutative analogues of the well known
identities for the Fibonacci-like integer sequences. Lastly, we derive some
binomial-sum identities for the Horadam quaternions as an application of the
matrix method.
\end{abstract}

\maketitle

\section{Introduction}

\label{section1}

The real quaternion algebra is defined by%
\begin{equation*}
H=\{a_{0}+a_{1}i+a_{2}j+a_{3}k:a_{l}\in 
\mathbb{R}
,l=\left\{ 0,1,2,3\right\} \},
\end{equation*}%
where the basis $\left\{ 1,i,j,k\right\} $ satisfies the following
multiplication rule:%
\begin{eqnarray*}
i^{2} &=&j^{2}=k^{2}=-1, \\
ij &=&-ji=k,\text{ }jk=-kj=i,\text{ }ki=-ik=j
\end{eqnarray*}

The addition and multiplication of two quaternions $\mathfrak{p}%
=a_{0}+a_{1}i+a_{2}j+a_{3}k$ and $\mathfrak{q}=b_{0}+b_{1}i+b_{2}j+b_{3}k$
are defined by%
\begin{equation*}
\mathfrak{p}+\mathfrak{q}:=\left( a_{0}+b_{0}\right) +\left(
a_{1}+b_{1}\right) i+\left( a_{2}+b_{2}\right) j+\left( a_{3}+b_{3}\right) k,
\end{equation*}%
\begin{eqnarray*}
\mathfrak{p}\mathfrak{q}&:=&a_{0}b_{0}-a_{1}b_{1}-a_{2}b_{2}-a_{3}b_{3} \\
&&+\left( a_{0}b_{1}+a_{1}b_{0}+a_{2}b_{3}-a_{3}b_{2}\right) i \\
&&+\left( a_{0}b_{2}+a_{2}b_{0}+a_{3}b_{1}-a_{1}b_{3}\right) j \\
&&+\left( a_{0}b_{3}+a_{3}b_{0}+a_{1}b_{2}-a_{2}b_{1}\right) k.\text{ \ \ \
\ \ \ \ \ \ \ \ \ \ \ \ \ \ \ \ \ \ \ \ \ }
\end{eqnarray*}

The multiplication of a quaternion by the real scalar $c$ is defined as%
\begin{equation*}
c\mathfrak{p}:=ca_{0}+ca_{1}i+ca_{2}j+ca_{3}k,
\end{equation*}

and the norm of a quaternion $\mathfrak{q}$ is defined by%
\begin{equation*}
N\left( \mathfrak{q}\right) :=\mathfrak{q}\overline{\mathfrak{q}}%
=a_{0}^{2}+a_{1}^{2}+a_{2}^{2}+a_{3}^{2}\text{ \ \ \ \ \ \ \ \ \ \ \ \ \ \ \
\ \ \ \ \ \ \ \ \ \ \ \ \ \ \ \ \ \ \ }
\end{equation*}%
where $\overline{\mathfrak{q}}=a_{0}-a_{1}i-a_{2}j-a_{3}k$ is the conjugate
of a generalized quaternion $\mathfrak{q}$. We note that, in general, $%
\overline{\mathfrak{p}}$ $\overline{\mathfrak{q}}\neq \overline{\mathfrak{p}%
\mathfrak{q}}$. Hamilton's book \cite{hamilton} serves as an excellent
reference to the properties of quaternions.

There has been increasing interest on Fibonacci-type quaternions which are
defined by using special integer sequences such as Fibonacci, Lucas, Pell,
Jacobsthal sequences, etc. The properties of these quaternion sequences have
been extensively studied by several researchers. In particular, Horadam \cite%
{horadam1} defined the Fibonacci quaternions and Lucas quaternions over the
real quaternion algebra $H$ as%
\begin{equation*}
Q_{n}=F_{n}+F_{n+1}i+F_{n+2}j+F_{n+3}k
\end{equation*}%
and%
\begin{equation*}
K_{n}=L_{n}+L_{n+1}i+L_{n+2}j+L_{n+3}k,
\end{equation*}%
where $F_{n}$ is the $n$-th Fibonacci number defined by the recurrence
relation $F_{n}=F_{n-1}+F_{n-2},\ \ n\geq 2$ with the initial conditions $%
F_{0}=0,F_{1}=1$ and $L_{n}$ is the $n$-th Lucas number which satisfies the
same recurrence relation as Fibonacci numbers but begins with the initial
conditions $L_{0}=2,L_{1}=1$. Hal\i c\i\ and Karatas \cite{halici}
generalized the Fibonacci quaternions to the Horadam quaternions as%
\begin{equation}
W_{n}=w_{n}+w_{n+1}i+w_{n+2}j+w_{n+3}k,  \label{h}
\end{equation}%
where $\left\{ w_{n}\right\} :=\left\{ w_{n}\left( w_{0},w_{1};p,q\right)
\right\} $ is the $n$th Horadam number \cite{horadam} defined by%
\begin{equation*}
w_{n}=pw_{n-1}+qw_{n-2},\text{ \ \ }n\geq 2
\end{equation*}%
with initial conditions $w_{0}$ and $w_{1}$. We note that $\left\{
w_{n}\left( 0,1;p,q\right) \right\} =\left\{ u_{n}\right\} $ is the $\left(
p,q\right) $-Fibonacci sequence and $\left\{ w_{n}\left( 2,p;p,q\right)
\right\} =\left\{ v_{n}\right\} $ is the $\left( p,q\right) $-Lucas
sequence. Similarly, if we take the initial conditions $w_{0}=0$ and $%
w_{1}=1 $, the equation (\ref{h}) is reduced to the $\left( p,q\right) $%
-Fibonacci quaternions \cite{ipek}:%
\begin{equation}
U_{n}=u_{n}+u_{n+1}i+u_{n+2}j+u_{n+3}k,  \label{u}
\end{equation}%
and if we take $w_{0}=2$ and $w_{1}=p,$ it gives the $\left( p,q\right) $%
-Lucas quaternions \cite{patel}:%
\begin{equation}
V_{n}=v_{n}+v_{n+1}i+v_{n+2}j+v_{n+3}k,  \label{v}
\end{equation}%
where $u_{n}$ and $v_{n}$ are the $n$th $\left( p,q\right) $-Fibonacci
quaternions and $\left( p,q\right) $-Lucas quaternions respectively.
Morales's paper \cite{morales1} serves as an excellent reference for the
properties of the $\left( p,q\right) $-Fibonacci quaternions and $\left(
p,q\right) $-Lucas quaternions. Also, the Horadam quaternions for negative
subscripts can be defined by%
\begin{equation*}
W_{-n}:=w_{-n}+w_{-n+1}i+w_{-n+2}j+w_{-n+3}k.
\end{equation*}%
For negative subscripts, we note that $\left( -q\right)
^{n}w_{-n}=w_{0}u_{n+1}-w_{1}u_{n}.$

On the other hand, the matrix method is also very useful to obtain some
algebraic properties of the Fibonacci numbers and Fibonacci quaternions. In
particular, the Cassini's identity can easily be obtained by computing the
determinant of the Fibonacci quaternion matrix $%
\begin{pmatrix}
Q_{2} & Q_{1} \\ 
Q_{1} & Q_{0}%
\end{pmatrix}%
$ which was first defined by Halici \cite{halici0}. Similar to the Fibonacci
quaternion matrix, Szynal-Liana and Wloch \cite{szynal1, szynal2} gave
matrix representations for the Pell quaternions and Jacobsthal quaternions.
More generally, Patel and Ray \cite{patel} introduced the $\left( p,q\right) 
$-Fibonacci quaternion matrix as follows:%
\begin{equation}
\mathbb{U}:=%
\begin{pmatrix}
U_{2} & qU_{1} \\ 
U_{1} & qU_{0}%
\end{pmatrix}%
\Rightarrow \mathbb{U}\mathbb{A}^{n-1}=%
\begin{pmatrix}
U_{n+1} & qU_{n} \\ 
U_{n} & qU_{n-1}%
\end{pmatrix}
\label{1}
\end{equation}%
where the matrix $\mathbb{A}$ satisfies the following matrix relation:%
\begin{equation}
\mathbb{A}:=%
\begin{pmatrix}
p & q \\ 
1 & 0%
\end{pmatrix}%
\Rightarrow \mathbb{A}^{n}=%
\begin{pmatrix}
u_{n+1} & qu_{n} \\ 
u_{n} & qu_{n-1}%
\end{pmatrix}%
.  \label{2}
\end{equation}%
Recently, Bitim \cite{bitim} introduced several new quaternion matrices and
derived some identities of Fibonacci and Lucas quaternions by using these
matrices. For details related to the matrix $\mathbb{A}$, see \cite{gould,
melham}. For additional references related to the Fibonacci-type matrices
and Fibonacci-type quaternion matrices, see the papers \cite{johnson,
morales, tan, t1, t2, t}.

It is well known that the quaternion multiplication is non-commutative. But
it is interesting that some results related to this property was obtained
incorrectly. Some of these were pointed out in our former work \cite{chaos1}%
. One of the main advantages of using the Fibonacci-type quaternion matrices
is that it prevents such kind of mathematical errors which may have occurred
in some published papers. Since the quaternion multiplication is
non-commutative, we get different results when we compute the determinant by
expanding along different rows or columns. By taking extra care of the
quaternion multiplication, we get correct results based on the determinant
properties of the quaternion matrices. It was also pointed out by Alves \cite%
{alves}.

This paper is organized as follows: In Section \ref{section2}, we derive
some matrix representations of the Horadam quaternions which generalize the
former Fibonacci-like quaternion matrices that are mentioned above. In
particular, we get two matrix identities, (\ref{4}) and (\ref{5}), to
generate the Horadam quaternions. Both of the identities can be used to
obtain some properties of the Horadam quaternions. The first matrix identity
(\ref{4}) is obtained by the product of the matrices whose entries are $%
\left( p,q\right) $-Fibonacci quaternions and Horadam numbers respectively.
The second matrix identity (\ref{5}) is obtained by the product of the
matrices whose entries are Horadam quaternions and $\left( p,q\right) $%
-Fibonacci numbers respectively. We derive several identities related to
these quaternions by using the matrix method. In Section \ref{section3}, we
obtain some binomial-sum identities for the Horadam quaternions as an
immediate application of the matrix method used in the previous section. Our
results generalize the identities derived by Halici \cite{halici0}, Ipek 
\cite{ipek}, Kesim and Polatli \cite{kesim}.

\section{Main results}

\label{section2}

It is known that, for the Horadam sequence $\left\{ w_{n}\right\} ,$ we have
the matrix relation 
\begin{equation}
\mathbb{T}:=%
\begin{pmatrix}
w_{2} & qw_{1} \\ 
w_{1} & qw_{0}%
\end{pmatrix}%
\Rightarrow \mathbb{T}\mathbb{A}^{n-1}=%
\begin{pmatrix}
w_{n+1} & qw_{n} \\ 
w_{n} & qw_{n-1}%
\end{pmatrix}%
.  \label{3}
\end{equation}

Considering the matrix equalities in (\ref{1}) and (\ref{3}), we have a
matrix representation of the Horadam quaternions as follows:

For $n\geq 1,$ we have%
\begin{equation}
\left( \mathbb{T}\mathbb{A}^{n-1}\right) \mathbb{U}=\mathbb{U}\left( \mathbb{%
T}\mathbb{A}^{n-1}\right) =%
\begin{pmatrix}
W_{n+2} & qW_{n+1} \\ 
W_{n+1} & qW_{n}%
\end{pmatrix}%
.  \label{4}
\end{equation}

On the other hand, by using induction we have another matrix representation
for the Horadam quaternions as follows:

For $n\geq 1,$ we have%
\begin{equation}
\mathbb{W}:=%
\begin{pmatrix}
W_{2} & qW_{1} \\ 
W_{1} & qW_{0}%
\end{pmatrix}%
\Rightarrow \mathbb{A}^{n}\mathbb{W}=\mathbb{W}\mathbb{A}^{n}=%
\begin{pmatrix}
W_{n+2} & qW_{n+1} \\ 
W_{n+1} & qW_{n}%
\end{pmatrix}%
.  \label{5}
\end{equation}

\begin{remark}
For $\{w_{n}(0,1;1,1)\}$ and $\{w_{n}(2,1;1,1)\}$ in (\ref{5}), we obtain
the matrix representations of Fibonacci and Lucas quaternions \cite[Theorem
1, Theorem 2]{bitim} respectively. Thus, the results of this paper reduce to
the identities in Bitim's paper \cite{bitim}.
\end{remark}

The following theorem gives the \textit{Cassini identities} for Horadam
quaternions. If we take the determinant on both sides of the matrix equation
(\ref{5}) by expanding along the first row and second row, we get (\ref{*})
and (\ref{**}) respectively. Similarly, if we take the determinant on both
sides of the matrix equation in (\ref{4}) by expanding along the first row
and second row, we get (\ref{***}) and (\ref{****}) respectively.

\begin{theorem}
\label{t1} For a nonnegative integer $n$, we have%
\begin{align}
W_{n+1}W_{n-1}-W_{n}^{2}& =\left( -q\right) ^{n-1}\left(
W_{2}W_{0}-W_{1}^{2}\right) ,  \label{*} \\
W_{n-1}W_{n+1}-W_{n}^{2}& =\left( -q\right) ^{n-1}\left(
W_{0}W_{2}-W_{1}^{2}\right) ,  \label{**} \\
W_{n+1}W_{n-1}-W_{n}^{2}& =\left( -q\right) ^{n-1}\left(
U_{2}U_{0}-U_{1}^{2}\right) \left( w_{1}^{2}-pw_{1}w_{0}-qw_{0}^{2}\right) ,
\label{***} \\
W_{n-1}W_{n+1}-W_{n}^{2}& =\left( -q\right) ^{n-1}\left(
U_{0}U_{2}-U_{1}^{2}\right) \left( w_{1}^{2}-pw_{1}w_{0}-qw_{0}^{2}\right) .
\label{****}
\end{align}
\end{theorem}

\begin{remark}
We note that the terms on the right hand side of the equations (\ref{***})
and (\ref{****}) can be expressed as%
\begin{align*}
U_{2}U_{0}-U_{1}^{2}& =-\left( -1+q-q^{2}+q^{3}+V_{0}-pq\left(
qi+p\allowbreak j-k\right) \right) , \\
U_{0}U_{2}-U_{1}^{2}& =-\left( -1+q-q^{2}+q^{3}+V_{0}+pq\left(
qi+p\allowbreak j-k\right) \right)
\end{align*}%
(see \cite{morales1}).
\end{remark}

\begin{theorem}
\label{t2} For integers $m,n\geq 1,$ the following equalities hold:%
\begin{equation}
w_{n}U_{m+1}+qw_{n-1}U_{m}=W_{n+m},\text{ \ \ \ \ \ \ \ \ \ \ \ \ \ \ \ \ \
\ \ \ \ }  \label{t21}
\end{equation}%
\begin{equation}
u_{n}W_{m+1}+qu_{n-1}W_{m}=W_{n+m},\text{ \ \ \ \ \ \ \ \ \ \ \ \ \ \ \ \ \
\ \ \ \ }  \label{t22}
\end{equation}%
\begin{equation}
W_{m+1}U_{n+1}+qW_{m}U_{n}=U_{2}W_{m+n}+qU_{1}W_{m+n-1},  \label{t23}
\end{equation}%
\begin{equation}
W_{m+1}W_{n+1}+qW_{m}W_{n}=W_{2}W_{m+n}+qW_{1}W_{m+n-1}.  \label{t24}
\end{equation}
\end{theorem}

\begin{proof}
By the matrix equality $\left( \mathbb{T}\mathbb{A}^{m+n-2}\right) \mathbb{U}%
=\left( \mathbb{T}\mathbb{A}^{n-1}\right) \left( \mathbb{A}^{m-1}\mathbb{U}%
\right) $ and the matrix equations (\ref{1}), (\ref{4}), we get%
\begin{equation*}
\begin{pmatrix}
W_{m+n+1} & qW_{m+n} \\ 
W_{m+n} & qW_{m+n-1}%
\end{pmatrix}%
=%
\begin{pmatrix}
w_{n+1} & qw_{n} \\ 
w_{n} & qw_{n-1}%
\end{pmatrix}%
\begin{pmatrix}
U_{m+1} & qU_{m} \\ 
U_{m} & qU_{m-1}%
\end{pmatrix}%
.
\end{equation*}%
Similarly, by the matrix equality $\mathbb{W}\mathbb{A}^{m+n-2}=\left( 
\mathbb{W}\mathbb{A}^{m-1}\right) \mathbb{A}^{n-1}$ and (\ref{5}), we get%
\begin{equation*}
\begin{pmatrix}
W_{m+n} & qW_{m+n-1} \\ 
W_{m+n-1} & qW_{m+n-2}%
\end{pmatrix}%
=%
\begin{pmatrix}
W_{m+1} & qW_{m} \\ 
W_{m} & qW_{m-1}%
\end{pmatrix}%
\begin{pmatrix}
u_{n} & qu_{n-1} \\ 
u_{n-1} & qu_{n-2}%
\end{pmatrix}%
.
\end{equation*}%
By comparing the corresponding entries in the matrices on both sides of the
equations, we get the results (\ref{t21}) and (\ref{t22}).

By the matrix equality $\mathbb{U}\left( \mathbb{T}\mathbb{A}^{m+n-3}\mathbb{%
U}\right) =\left( \mathbb{U}\mathbb{T}\mathbb{A}^{m-2}\right) \left( \mathbb{%
U}\mathbb{A}^{n-1}\right) $ and the matrix equations (\ref{1}), (\ref{4}),
we get 
\begin{equation*}
\begin{pmatrix}
U_{2} & qU_{1} \\ 
U_{1} & qU_{0}%
\end{pmatrix}%
\begin{pmatrix}
W_{m+n} & qW_{m+n-1} \\ 
W_{m+n-1} & qW_{m+n-2}%
\end{pmatrix}%
=%
\begin{pmatrix}
W_{m+1} & qW_{m} \\ 
W_{m} & qW_{m-1}%
\end{pmatrix}%
\begin{pmatrix}
U_{n+1} & qU_{n} \\ 
U_{n} & qU_{n-1}%
\end{pmatrix}%
.
\end{equation*}%
Similarly, by the matrix equality $\mathbb{W}\left( \mathbb{A}^{m+n-2}%
\mathbb{W}\right) =\left( \mathbb{W}\mathbb{A}^{m-1}\right) \left( \mathbb{W}%
\mathbb{A}^{n-1}\right) $ and (\ref{5}), we get 
\begin{equation*}
\begin{pmatrix}
W_{2} & qW_{1} \\ 
W_{1} & qW_{0}%
\end{pmatrix}%
\begin{pmatrix}
W_{m+n} & qW_{m+n-1} \\ 
W_{m+n-1} & qW_{m+n-2}%
\end{pmatrix}%
=%
\begin{pmatrix}
W_{m+1} & qW_{m} \\ 
W_{m} & qW_{m-1}%
\end{pmatrix}%
\begin{pmatrix}
W_{n+1} & qW_{n} \\ 
W_{n} & qW_{n-1}%
\end{pmatrix}%
.
\end{equation*}%
By comparing the corresponding entries in the matrices on both sides of the
equation, we get the result (\ref{t23}) and (\ref{t24}).
\end{proof}

We note that, although the first matrix identity (\ref{4}) can be used to
obtain identities involving both $\left( p,q\right) $-Fibonacci quaternions
and Horadam quaternions, the second matrix identity (\ref{5}) gives
identities involving only Horadam quaternions.

As corollaries, we state two quadratic identities for Horadam quaternions
below. 

\begin{corollary}
\label{corollary0.5} Let $n$ be a positive integer. The following equality
holds: 
\begin{equation*}
W_{n+1}^2 + q W_n^2 = W_1 W_{2n+1} + q W_0 W_{2n}.
\end{equation*}
\end{corollary}

\begin{proof}
By Theorem \ref{t2}, we have 
\begin{align}
W_{n+1}^{2}+qW_{n}^{2}& =W_{2}W_{2n}+qW_{1}W_{2n-1}  \notag \\
& =\left( pW_{1}+qW_{0}\right) W_{2n}+qW_{1}W_{2n-1}  \notag \\
& =qW_{0}W_{2n}+W_{1}\left( pW_{2n}+qW_{2n-1}\right)  \notag \\
& =W_{1}W_{2n+1}+qW_{0}W_{2n}.  \notag
\end{align}
\end{proof}

\begin{corollary}
\label{corollary0.6} Let $n$ be a positive integer. The following equality
holds: 
\begin{equation*}
W_{n+1}^{2}-q^{2}W_{n-1}^{2}=p\left( W_{1}W_{2n}+qW_{0}W_{2n-1}\right) .
\end{equation*}
\end{corollary}

\begin{proof}
First, we have 
\begin{equation}
W_{n+1}^{2}-q^{2}W_{n-1}^{2}=\left( W_{n+1}^{2}+qW_{n}^{2}\right) -q\left(
W_{n}^{2}+qW_{n-1}^{2}\right) .  \label{equation90}
\end{equation}%
Then, we do the following computations 
\begin{align}
W_{n+1}^{2}+qW_{n}^{2}& =%
\begin{pmatrix}
W_{n+1} & qW_{n}%
\end{pmatrix}%
\begin{pmatrix}
W_{n+1} \\ 
W_{n}%
\end{pmatrix}
\notag \\
& =%
\begin{pmatrix}
W_{1} & qW_{0}%
\end{pmatrix}%
\mathbb{A}^{n}\mathbb{A}^{n}%
\begin{pmatrix}
W_{1} \\ 
W_{0}%
\end{pmatrix}
\notag \\
& =%
\begin{pmatrix}
W_{1} & qW_{0}%
\end{pmatrix}%
\mathbb{A}^{2n}%
\begin{pmatrix}
W_{1} \\ 
W_{0}%
\end{pmatrix}%
.  \label{equation91}
\end{align}%
Similarly, we also have 
\begin{equation}
W_{n}^{2}+qW_{n-1}^{2}=%
\begin{pmatrix}
W_{1} & qW_{0}%
\end{pmatrix}%
\mathbb{A}^{2n-2}%
\begin{pmatrix}
W_{1} \\ 
W_{0}%
\end{pmatrix}%
.  \label{equation92}
\end{equation}%
By (\ref{equation90}), (\ref{equation91}) and (\ref{equation92}), we have 
\begin{align}
W_{n+1}^{2}-q^{2}W_{n-1}^{2}& =%
\begin{pmatrix}
W_{1} & qW_{0}%
\end{pmatrix}%
\left( \mathbb{A}^{2n}-q\mathbb{A}^{2n-2}\right) 
\begin{pmatrix}
W_{1} \\ 
W_{0}%
\end{pmatrix}
\notag \\
& =%
\begin{pmatrix}
W_{1} & qW_{0}%
\end{pmatrix}%
\mathbb{A}^{2n-2}\left( \mathbb{A}^{2}-qI\right) 
\begin{pmatrix}
W_{1} \\ 
W_{0}%
\end{pmatrix}
\notag \\
& =p%
\begin{pmatrix}
W_{1} & qW_{0}%
\end{pmatrix}%
\mathbb{A}^{2n-2}\mathbb{A}%
\begin{pmatrix}
W_{1} \\ 
W_{0}%
\end{pmatrix}
\label{equation93} \\
& =p\left( W_{1}W_{2n}+qW_{0}W_{2n-1}\right) .  \notag
\end{align}%
We note that the equality (\ref{equation93}) is due to the Cayley-Hamilton
theorem: 
\begin{equation*}
\mathbb{A}^{2}-p\mathbb{A}-qI=%
\begin{pmatrix}
0 & 0 \\ 
0 & 0%
\end{pmatrix}%
.
\end{equation*}
\end{proof}

By Theorem \ref{t2}, we have the following matrix equations for the Horadam
quaternions $W_{n}$: 
\begin{align}
\begin{pmatrix}
W_{n+m+1} \\ 
W_{n+m}%
\end{pmatrix}%
& =%
\begin{pmatrix}
w_{n+1} & qw_{n} \\ 
w_{n} & qw_{n-1}%
\end{pmatrix}%
\begin{pmatrix}
U_{m+1} \\ 
U_{m}%
\end{pmatrix}%
,  \label{equation100} \\
\begin{pmatrix}
W_{n+m} \\ 
W_{n}%
\end{pmatrix}%
& =%
\begin{pmatrix}
u_{m} & qu_{m-1} \\ 
0 & 1%
\end{pmatrix}%
\begin{pmatrix}
W_{n+1} \\ 
W_{n}%
\end{pmatrix}%
.  \label{equation101}
\end{align}

\noindent We look at another Horadam sequence $\{z_{n}\}:=%
\{w_{n}(z_{0},z_{1};p;q)\}$. That is, 
\begin{equation*}
z_{n}=pz_{n-1}+qz_{n-2},\quad n\geq 2
\end{equation*}%
with initial conditions $z_{0}$ and $z_{1}$. We get another set of Horadam
quaternions as follows: 
\begin{equation*}
Z_{n}=z_{n}+z_{n+1}i+z_{n+2}j+z_{n+3}k.
\end{equation*}%
It is clear that the Horadam quaternions $Z_{n}$ satisfy the same set of
matrix equations (\ref{equation100}) and (\ref{equation101}) by Theorem \ref%
{t2}. The following theorem is a generalization of the $Catalan$ $identity$
for Fibonacci sequence to the cases that involve two sets of Horadam
quaternions $W_{n}$ and $Z_{n}$.

\begin{theorem}
\label{theorem100} Let $n$, $r$, and $s$ be positive integers. Let $W_n$ and 
$Z_n$ be two sets of Horadam quaternions defined by using the Horadam
sequences $\{w_n\}=\{ w_n(w_0, w_1;p,q)\}$ and $\{ z_n\}= \{ w_n (z_0, z_1;
p,q)\}$ respectively. Then, 
\begin{equation*}
Z_{n+r} W_{n+s} - Z_n W_{n+r+s} =(-q)^n u_r \left( Z_1 W_s - Z_0
W_{s+1}\right).
\end{equation*}
\end{theorem}

\begin{proof}
By the equation (\ref{equation101}), we do the following computation:%
\begin{equation}
\begin{pmatrix}
Z_{n+r} & Z_{n}%
\end{pmatrix}%
=%
\begin{pmatrix}
Z_{n+1} & Z_{n}%
\end{pmatrix}%
\begin{pmatrix}
u_{r} & 0 \\ 
qu_{r-1} & 1%
\end{pmatrix}%
=%
\begin{pmatrix}
Z_{1} & Z_{0}%
\end{pmatrix}%
\begin{pmatrix}
p & 1 \\ 
q & 0%
\end{pmatrix}%
^{n}%
\begin{pmatrix}
u_{r} & 0 \\ 
qu_{r-1} & 1%
\end{pmatrix}%
,  \label{equation102}
\end{equation}%
\begin{equation}
\begin{pmatrix}
W_{n+s} \\ 
-W_{n+r+s}%
\end{pmatrix}%
=%
\begin{pmatrix}
1 & 0 \\ 
-qu_{r-1} & u_{r}%
\end{pmatrix}%
\begin{pmatrix}
W_{n+s} \\ 
-W_{n+s+1}%
\end{pmatrix}%
=%
\begin{pmatrix}
1 & 0 \\ 
-qu_{r-1} & u_{r}%
\end{pmatrix}%
\begin{pmatrix}
0 & -1 \\ 
-q & p%
\end{pmatrix}%
^{n+s}%
\begin{pmatrix}
W_{0} \\ 
-W_{1}%
\end{pmatrix}%
.  \label{equation103}
\end{equation}
We also note that 
\begin{align}
& 
\begin{pmatrix}
p & 1 \\ 
q & 0%
\end{pmatrix}%
^{n}%
\begin{pmatrix}
u_{r} & 0 \\ 
qu_{r-1} & 1%
\end{pmatrix}%
\begin{pmatrix}
1 & 0 \\ 
-qu_{r-1} & u_{r}%
\end{pmatrix}%
\begin{pmatrix}
0 & -1 \\ 
-q & p%
\end{pmatrix}%
^{n+s}  \notag \\
& =%
\begin{pmatrix}
p & 1 \\ 
q & 0%
\end{pmatrix}%
^{n}\left( u_{r}I\right) 
\begin{pmatrix}
0 & -1 \\ 
-q & p%
\end{pmatrix}%
^{n+s}  \notag \\
& =u_{r}%
\begin{pmatrix}
p & 1 \\ 
q & 0%
\end{pmatrix}%
^{n}%
\begin{pmatrix}
0 & -1 \\ 
-q & p%
\end{pmatrix}%
^{n}%
\begin{pmatrix}
0 & -1 \\ 
-q & p%
\end{pmatrix}%
^{s}  \notag \\
& =u_{r}\left( (-q)^{n}I\right) 
\begin{pmatrix}
0 & -1 \\ 
-q & p%
\end{pmatrix}%
^{s}=(-q)^{n}u_{r}%
\begin{pmatrix}
0 & -1 \\ 
-q & p%
\end{pmatrix}%
^{s}.  \label{equation104}
\end{align}%
By (\ref{equation102}), (\ref{equation103}) and (\ref{equation104}), we get 
\begin{align}
Z_{n+r}W_{n+s}-Z_{n}W_{n+r+s}& =%
\begin{pmatrix}
Z_{n+r} & Z_{n}%
\end{pmatrix}%
\begin{pmatrix}
W_{n+s} \\ 
-W_{n+r+s}%
\end{pmatrix}
\notag \\
& =(-q)^{n}u_{r}%
\begin{pmatrix}
Z_{1} & Z_{0}%
\end{pmatrix}%
\begin{pmatrix}
0 & -1 \\ 
-q & p%
\end{pmatrix}%
^{s}%
\begin{pmatrix}
W_{0} \\ 
-W_{1}%
\end{pmatrix}
\label{equation105} \\
& =(-q)^{n}u_{r}%
\begin{pmatrix}
Z_{1} & Z_{0}%
\end{pmatrix}%
\begin{pmatrix}
W_{s} \\ 
-W_{s+1}%
\end{pmatrix}
\notag \\
& =(-q)^{n}u_{r}\left( Z_{1}W_{s}-Z_{0}W_{s+1}\right) .  \notag
\end{align}
\end{proof}

\begin{remark}
It is worthwhile to note that Tangboonduangjit and Thanatipanonda \cite[%
Proposition 1]{thotsaporn} proved the same identity as in Theorem \ref%
{theorem100} for integer sequences satisfying a general second-order
recurrence relation with constant coefficients.
\end{remark}

We have the following corollary immediately.

\begin{corollary}
\label{corollary1} For positive integers $n,r,$ and $s,$ we have%
\begin{eqnarray*}
U_{n+r}U_{n+s}-U_{n}U_{n+r+s} &=&(-q)^{n}u_{r}\left(
U_{1}U_{s}-U_{0}U_{s+1}\right) , \\
U_{n+r}W_{n+s}-U_{n}W_{n+r+s} &=&(-q)^{n}u_{r}\left(
U_{1}W_{s}-U_{0}W_{s+1}\right) , \\
W_{n+r}U_{n+s}-W_{n}U_{n+r+s} &=&(-q)^{n}u_{r}\left(
W_{1}U_{s}-W_{0}U_{s+1}\right) , \\
W_{n+r}W_{n+s}-W_{n}W_{n+r+s} &=&(-q)^{n}u_{r}\left(
W_{1}W_{s}-W_{0}W_{s+1}\right) .
\end{eqnarray*}
\end{corollary}

\begin{remark}
If we replace $n$ by $n-1$ and set $r=1$, $s=1$, we get back the Cassini's
identity (\ref{**}).
\end{remark}

We define a \textit{commutator bracket} between the Horadam quaternions $%
Z_{n}$ and $W_{n}$ by 
\begin{equation*}
\lbrack Z_{n},W_{n}]_{0}:=Z_{1}W_{0}-Z_{0}W_{1}.
\end{equation*}%
We note that $[W_{n},W_{n}]_{0}\neq 0$ since multiplication is
non-commutative for quaternions. Also, we define the following function for $%
Z_{n}$ and $W_{n}$: 
\begin{equation*}
\Delta (Z_{n},W_{n}):=Z_{1}W_{1}-pZ_{0}W_{1}-qZ_{0}W_{0}.
\end{equation*}%
It is clear that the commutator bracket and the function $\Delta $ can be
defined on any integer sequences also. For examples, we have $%
[w_{n},w_{n}]_{0}=0$ for any Horadam sequence $\{w_{n}\}$ and%
\begin{eqnarray*}
\Delta (F_{n},F_{n}) &=&F_{1}F_{1}-F_{0}F_{1}-F_{0}F_{0}=1, \\
\Delta (L_{n},L_{n}) &=&L_{1}L_{1}-L_{0}L_{1}-L_{0}L_{0}=-5, \\
\Delta (F_{n},L_{n}) &=&F_{1}L_{1}-F_{0}L_{1}-F_{0}L_{0}=1, \\
\Delta (L_{n},F_{n}) &=&L_{1}F_{1}-L_{0}F_{1}-L_{0}F_{0}=-1, \\
\Delta (w_{n},w_{n}) &=&w_{1}^{2}-pw_{0}w_{1}-qw_{0}^{2}.
\end{eqnarray*}%
Waddill \cite[equation (21)]{waddill} showed the following result for
Horadam sequences $\{w_{n}\}$: 
\begin{equation}
w_{n+r}w_{n+s}-w_{n}w_{n+r+s}=(-q)^{n}u_{r}u_{s}\Delta (w_{n},w_{n}).
\label{equation106}
\end{equation}%
\noindent The matrix computation we did in the proof of Theorem \ref%
{theorem100} can be done in an alternative way such that the expression $%
Z_{n+r}W_{n+s}-Z_{n}W_{n+r+s}$ does not depend on $Z_{n}$ and $W_{n}$ for $%
n\geq 2$ . We present it as a corollary.

\begin{corollary}
\label{corollary2} Let $n$, $r$, and $s$ be positive integers. Let $W_{n}$
and $Z_{n}$ be two sets of Horadam quaternions defined by using the Horadam
sequences $\{w_{n}\}=\{w_{n}(w_{0},w_{1};p,q)\}$ and $\{z_{n}\}=%
\{w_{n}(z_{0},z_{1};p,q)\}$ respectively. Then, 
\begin{equation*}
Z_{n+r}W_{n+s}-Z_{n}W_{n+r+s}=(-q)^{n}u_{r}\left(
qu_{s-1}[Z_{n},W_{n}]_{0}+u_{s}\Delta (Z_{n},W_{n})\right) .
\end{equation*}
\end{corollary}

\begin{proof}
By the equation (\ref{equation105}), we have the following computation: 
\begin{align*}
& (-q)^{n}u_{r}%
\begin{pmatrix}
Z_{1} & Z_{0}%
\end{pmatrix}%
\begin{pmatrix}
0 & -1 \\ 
-q & p%
\end{pmatrix}%
^{s}%
\begin{pmatrix}
W_{0} \\ 
-W_{1}%
\end{pmatrix}
\\
& =(-q)^{n}u_{r}%
\begin{pmatrix}
Z_{1} & Z_{0}%
\end{pmatrix}%
\begin{pmatrix}
qu_{s-1} & -u_{s} \\ 
-qu_{s} & u_{s+1}%
\end{pmatrix}%
\begin{pmatrix}
W_{0} \\ 
-W_{1}%
\end{pmatrix}
\\
& =(-q)^{n}u_{r}\left(
qu_{s-1}Z_{1}W_{0}-qu_{s}Z_{0}W_{0}+u_{s}Z_{1}W_{1}-u_{s+1}Z_{0}W_{1}\right)
\\
& =(-q)^{n}u_{r}\left(
qu_{s-1}Z_{1}W_{0}-qu_{s}Z_{0}W_{0}+u_{s}Z_{1}W_{1}-(pu_{s}+qu_{s-1})Z_{0}W_{1}\right)
\\
& =(-q)^{n}u_{r}\big(qu_{s-1}\left( Z_{1}W_{0}-Z_{0}W_{1}\right)
+u_{s}\left( Z_{1}W_{1}-pZ_{0}W_{1}-qZ_{0}W_{0}\right) \big)
\end{align*}%
and hence the desired result.
\end{proof}

\begin{remark}
By Corollary \ref{corollary2}, we have 
\begin{equation*}
W_{n+r}W_{n+s}-W_{n}W_{n+r+s}=(-q)^{n}u_{r}\left(
qu_{s-1}[W_{n},W_{n}]_{0}+u_{s}\Delta (W_{n},W_{n})\right) .
\end{equation*}%
It is the \textit{non-commutative} version of the identity (\ref{equation106}%
) as we have a non-commutative multiplication rule for the Horadam
quaternions.
\end{remark}

On the other hand, we have yet another matrix computation for the expression 
$W_{n+r}W_{n+s}-W_{n}W_{n+r+s}$ due to equation (\ref{equation100}). The
theorem below is a generalization of Cassini's identities stated in
equations (\ref{***}) and (\ref{****}).

\begin{theorem}
\label{theorem101} Let $n$, $r$, and $s$ be positive integers. Let $W_{n}$
be the Horadam quaternions defined by using the Horadam sequences $%
\{w_{n}\}=\{w_{n}(w_{0},w_{1};p,q)\}$. Then, 
\begin{equation*}
W_{n+r}W_{n+s}-W_{n}W_{n+r+s}=(-q)^{n}u_{r}\left(
U_{1}U_{s}-U_{0}U_{s+1}\right) \Delta (w_{n},w_{n}).
\end{equation*}
\end{theorem}

\begin{proof}
By equations (\ref{equation100}) and (\ref{equation101}), we have the
following matrix computation, 
\begin{align}
\begin{pmatrix}
W_{n+r} & W_{n}%
\end{pmatrix}%
& =%
\begin{pmatrix}
W_{n+1} & W_{n}%
\end{pmatrix}%
\begin{pmatrix}
u_{r} & 0 \\ 
qu_{r-1} & 1%
\end{pmatrix}
\notag \\
& =%
\begin{pmatrix}
U_{1} & U_{0}%
\end{pmatrix}%
\begin{pmatrix}
w_{n+1} & w_{n} \\ 
qw_{n} & qw_{n-1}%
\end{pmatrix}%
\begin{pmatrix}
u_{r} & 0 \\ 
qu_{r-1} & 1%
\end{pmatrix}%
,  \label{equation107} \\
\begin{pmatrix}
W_{n+s} \\ 
-W_{n+r+s}%
\end{pmatrix}%
& =%
\begin{pmatrix}
1 & 0 \\ 
-qu_{r-1} & u_{r}%
\end{pmatrix}%
\begin{pmatrix}
W_{n+s} \\ 
-W_{n+s+1}%
\end{pmatrix}
\notag \\
& =%
\begin{pmatrix}
1 & 0 \\ 
-qu_{r-1} & u_{r}%
\end{pmatrix}%
\begin{pmatrix}
qw_{n-1} & -w_{n} \\ 
-qw_{n} & w_{n+1}%
\end{pmatrix}%
\begin{pmatrix}
U_{s} \\ 
-U_{s+1}%
\end{pmatrix}%
.  \label{equation108}
\end{align}%
We also note that 
\begin{align}
& 
\begin{pmatrix}
w_{n+1} & w_{n} \\ 
qw_{n} & qw_{n-1}%
\end{pmatrix}%
\begin{pmatrix}
u_{r} & 0 \\ 
qu_{r-1} & 1%
\end{pmatrix}%
\begin{pmatrix}
1 & 0 \\ 
-qu_{r-1} & u_{r}%
\end{pmatrix}%
\begin{pmatrix}
qw_{n-1} & -w_{n} \\ 
-qw_{n} & w_{n+1}%
\end{pmatrix}
\notag \\
& =%
\begin{pmatrix}
w_{n+1} & w_{n} \\ 
qw_{n} & qw_{n-1}%
\end{pmatrix}%
\left( u_{r}I\right) 
\begin{pmatrix}
qw_{n-1} & -w_{n} \\ 
-qw_{n} & w_{n+1}%
\end{pmatrix}
\notag \\
& =u_{r}%
\begin{pmatrix}
w_{n+1} & w_{n} \\ 
qw_{n} & qw_{n-1}%
\end{pmatrix}%
\begin{pmatrix}
qw_{n-1} & -w_{n} \\ 
-qw_{n} & w_{n+1}%
\end{pmatrix}
\notag \\
& =u_{r}%
\begin{pmatrix}
q(w_{n+1}w_{n-1}-w_{n}^{2}) & 0 \\ 
0 & q(-w_{n}^{2}+w_{n-1}w_{n+1})%
\end{pmatrix}
\notag \\
& =(-q)u_{r}(w_{n}^{2}-w_{n+1}w_{n-1})I  \notag \\
& =(-q)^{n}u_{r}\left( w_{1}^{2}-w_{0}w_{2}\right) I.  \label{equation109}
\end{align}%
The last equality is due to a general Catalan identity for Horadam sequences
(see Tangboonduangjit and Thanatipanonda \cite[Proposition 1]{thotsaporn}).
We combine (\ref{equation107}), (\ref{equation108}) and (\ref{equation109})
to get 
\begin{align*}
W_{n+r}W_{n+s}-W_{n}W_{n+r+s}& =%
\begin{pmatrix}
W_{n+r} & W_{n}%
\end{pmatrix}%
\begin{pmatrix}
W_{n+s} \\ 
-W_{n+r+s}%
\end{pmatrix}
\\
& =(-q)^{n}u_{r}\left( w_{1}^{2}-w_{0}w_{2}\right) 
\begin{pmatrix}
U_{1} & U_{0}%
\end{pmatrix}%
\begin{pmatrix}
U_{s} \\ 
U_{s+1}%
\end{pmatrix}
\\
& =(-q)^{n}u_{r}\left( w_{1}^{2}-w_{0}w_{2}\right) \left(
U_{1}U_{s}-U_{0}U_{s+1}\right) \\
& =(-q)^{n}u_{r}\left( w_{1}^{2}-w_{0}(pw_{1}+qw_{0})\right) \left(
U_{1}U_{s}-U_{0}U_{s+1}\right) \\
& =(-q)^{n}u_{r}\left( U_{1}U_{s}-U_{0}U_{s+1}\right) \Delta (w_{n},w_{n}).
\end{align*}
\end{proof}

\begin{remark}
If we replace $n$ by $n-1$ and set $r=s=1$ in Theorem \ref{theorem101}, then
we get the identities (\ref{***}) and (\ref{****}).
\end{remark}


Finally, we will derive identities involving $\left( p,q\right) $-Fibonacci
and $\left( p,q\right) $-Lucas quaternions. We need another matrix identity
as follows:

For $n\geq 1$, by induction, we have 
\begin{equation}
\mathbb{K}:=%
\begin{pmatrix}
V_{1} & dU_{1} \\ 
U_{1} & V_{1}%
\end{pmatrix}%
\Rightarrow \mathbb{KB}^{n-1}=2^{n-1}%
\begin{pmatrix}
V_{n} & dU_{n} \\ 
U_{n} & V_{n}%
\end{pmatrix}
\label{6}
\end{equation}%
where the matrix $\mathbb{B}$ satisfies the following matrix relation:%
\begin{equation}
\mathbb{B}:=%
\begin{pmatrix}
p & d \\ 
1 & p%
\end{pmatrix}%
\Rightarrow \mathbb{B}^{n}=2^{n-1}%
\begin{pmatrix}
v_{n} & du_{n} \\ 
u_{n} & v_{n}%
\end{pmatrix}
\label{66}
\end{equation}%
and $d:=p^{2}+4q.$

By using the matrix equalities (\ref{6}) and (\ref{66}), we obtain the
following results.

\begin{theorem}
\label{t4} For $n\geq 1,$ we have%
\begin{eqnarray}
V_{n}^{2}-dU_{n}^{2} &=&\left( -q\right) ^{n-1}\left(
V_{1}^{2}-dU_{1}^{2}\right)  \notag \\
&=&4\left( -q\right) ^{n}\left( V_{0}+\left( -1+q-q^{2}+q^{3}\right) \right)
,  \label{t41}
\end{eqnarray}%
\begin{equation}
v_{m}V_{n}+du_{m}U_{n}=2V_{m+n},\text{ \ \ \ \ \ \ \ \ \ \ \ \ \ \ \ \ \ \ \
\ \ \ \ \ \ \ \ \ }  \label{t42}
\end{equation}%
\begin{equation}
u_{m}V_{n}+dv_{m}U_{n}=2U_{m+n},\text{ \ \ \ \ \ \ \ \ \ \ \ \ \ \ \ \ \ \ \
\ \ \ \ \ \ \ \ \ }  \label{t43}
\end{equation}%
\begin{equation}
V_{m}V_{n}+dU_{m}U_{n}=V_{1}V_{m+n-1}+dU_{1}U_{m+n-1},\text{ \ \ \ \ \ }
\label{t44}
\end{equation}%
\begin{equation}
U_{m}V_{n}+V_{m}U_{n}=U_{1}V_{m+n-1}+V_{1}U_{m+n-1}.\text{ \ \ \ \ \ \ \ \ \
\ }  \label{t45}
\end{equation}
\end{theorem}

\begin{proof}
By taking determinant on both sides of the matrix equality (\ref{6}), we
obtain (\ref{t41}) (see also Morales's paper \cite{morales1}).

By considering the matrix equality $\mathbb{KB}^{m+n-1}=\left( \mathbb{KB}%
^{n-1}\right) \mathbb{B}^{m}$, we get the identities (\ref{t42}) and (\ref%
{t43}).

By the matrix equality $\mathbb{K}\left( \mathbb{B}^{m+n-2}\mathbb{K}\right)
=\mathbb{K}\left( \mathbb{K}\mathbb{B}^{m+n-2}\right) =\left( \mathbb{K}%
\mathbb{B}^{m-1}\right) \left( \mathbb{K}\mathbb{B}^{n-1}\right) $, we have%
\begin{equation*}
\begin{pmatrix}
V_{1} & dU_{1} \\ 
U_{1} & V_{1}%
\end{pmatrix}%
\begin{pmatrix}
V_{m+n-1} & dU_{m+n-1} \\ 
U_{m+n-1} & V_{m+n-1}%
\end{pmatrix}%
=%
\begin{pmatrix}
V_{m} & dU_{m} \\ 
U_{m} & V_{m}%
\end{pmatrix}%
\begin{pmatrix}
V_{n} & dU_{n} \\ 
U_{n} & V_{n}%
\end{pmatrix}%
\end{equation*}%
and hence 
\begin{align*}
& 
\begin{pmatrix}
V_{1}V_{m+n-1}+dU_{1}U_{m+n-1} & d\left( V_{1}U_{m+n-1}+U_{1}V_{m+n-1}\right)
\\ 
U_{1}V_{m+n-1}+V_{1}U_{m+n-1} & dU_{1}U_{m+n-1}+V_{1}V_{m+n-1}%
\end{pmatrix}
\\
& =%
\begin{pmatrix}
V_{m}V_{n}+dU_{m}U_{n} & d\left( V_{m}U_{n}+U_{m}V_{n}\right) \\ 
U_{m}V_{n}+V_{m}U_{n} & dU_{m}U_{n}+V_{m}V_{n}%
\end{pmatrix}%
.
\end{align*}%
By comparing the corresponding entries in the matrices on both sides of the
equation, we get the identities (\ref{t44}) and (\ref{t45}).
\end{proof}

\begin{remark}
We note that, for $p=q=1,$ the first result reduces to the identity%
\begin{equation*}
K_{n}^{2}-5Q_{n}^{2}=4\left( -1\right) ^{n}\left( 2+i+3j+4k\right) .
\end{equation*}
\end{remark}

\section{Some binomial-sum identities for Horadam quaternions}

\label{section3}

For any positive integer $n$, by (\ref{5}), we have 
\begin{align}
\mathbb{A}^n 
\begin{pmatrix}
W_1 \\ 
W_0%
\end{pmatrix}
&= 
\begin{pmatrix}
p & q \\ 
1 & 0%
\end{pmatrix}%
^n 
\begin{pmatrix}
W_1 \\ 
W_0%
\end{pmatrix}%
=%
\begin{pmatrix}
W_{n+1} \\ 
W_n%
\end{pmatrix}%
.  \label{equation3.1}
\end{align}%
Also, by induction on $n$, we have 
\begin{align}
\mathbb{A}^{-n} 
\begin{pmatrix}
W_1 \\ 
W_0%
\end{pmatrix}
&= \frac{1}{q^n} 
\begin{pmatrix}
0 & q \\ 
1 & -p%
\end{pmatrix}%
^n 
\begin{pmatrix}
W_1 \\ 
W_0%
\end{pmatrix}
=%
\begin{pmatrix}
W_{1-n} \\ 
W_{-n}%
\end{pmatrix}%
.  \label{equation3.1.2}
\end{align}%
By the Cayley-Hamilton theorem, we have 
\begin{align}
\mathbb{A}^2 &= p \mathbb{A} +q I.  \label{equation3.2}
\end{align}%
We state some binomial--sum identities for Horadam quaternions below.

\begin{theorem}
\label{theorem3.1} Let $n$ and $k$ be non-negative integers. Then, 
\begin{align*}
\sum_{j=0}^n \binom{n}{j} p^j q^{n-j} W_{j+k} &= W_{2n+k}.
\end{align*}
\end{theorem}

\begin{proof}
By (\ref{equation3.2}), we have 
\begin{equation*}
\mathbb{A}^{2n+k}=\mathbb{A}^{k}\mathbb{A}^{2n}=\mathbb{A}^{k}(p\mathbb{A}%
+qI)^{n}=\mathbb{A}^{k}\sum_{j=0}^{n}\binom{n}{j}p^{j}q^{n-j}\mathbb{A}%
^{j}=\sum_{j=0}^{n}\binom{n}{j}p^{j}q^{n-j}\mathbb{A}^{j+k}.
\end{equation*}%
Hence, we have 
\begin{equation*}
\mathbb{A}^{2n+k}%
\begin{pmatrix}
W_{1} \\ 
W_{0}%
\end{pmatrix}%
=\left( \sum_{j=0}^{n}\binom{n}{j}p^{j}q^{n-j}\mathbb{A}^{j+k}\right) 
\begin{pmatrix}
W_{1} \\ 
W_{0}%
\end{pmatrix}%
.
\end{equation*}%
So, we obtain the result by (\ref{equation3.1}) and comparing the second
entries in the matrices on both sides of the equation.
\end{proof}

\begin{remark}
If we set $k=0$ in Theorem \ref{theorem3.1}, for the $(p,q)$-Fibonacci
quaternions defined by the $(p,q)$-Fibonacci sequence $\{w_{n}(0,1;p,q)\}$,
we get the result proved by Ipek \cite[Theorem 2.7]{ipek}; for the Horadam
quaternions defined by the Horadam sequences $\{w_{n}(0,1;p,1)\}$ and $%
\{w_{n}(2,p;p,1)\}$, we get the identities shown by Kesim and Polatli \cite[%
Theorem 2.4]{kesim}; for the Fibonacci quaternions, we get the first
identity of Theorem 3.5 in Halici's paper \cite{halici0}.
\end{remark}

\begin{theorem}
\label{theorem3.2} Let $n$ and $k$ be non-negative integers. Then, 
\begin{align*}
\sum_{j=0}^n \binom{n}{j} (-1)^j p^{n-j} W_{j+k} &= (-q)^n W_{-n+k}.
\end{align*}
\end{theorem}

\begin{proof}
By (\ref{equation3.2}), we have 
\begin{align}
\mathbb{A} ( pI- \mathbb{A} ) &= -qI  \notag \\
(pI-\mathbb{A}) &= -q \mathbb{A}^{-1} .  \notag
\end{align}%
Hence, we have 
\begin{equation*}
(-q)^n \mathbb{A}^{-n+k}=\mathbb{A}^k (pI-\mathbb{A})^n =\mathbb{A}^k
\sum_{j=0}^n \binom{n}{j}(-1)^j p^{n-j}\mathbb{A}^{j}= \sum_{j=0}^n \binom{n%
}{j}(-1)^j p^{n-j}\mathbb{A}^{j+k}.
\end{equation*}
So, we obtain the result by (\ref{equation3.1.2}) and comparing the second
entries in the matrices on both sides of the equation.
\end{proof}

\begin{remark}
If we set $k=0$ in Theorem \ref{theorem3.2} for the Fibonacci quaternions,
then we get the second identity of Theorem 3.5 in Halici's paper \cite%
{halici0}.
\end{remark}

\begin{theorem}
\label{theorem3.3} Let $n$ and $k$ be non-negative integers. Then, 
\begin{equation*}
\sum_{j=0}^{n}\binom{n}{j}q^{n-j}W_{2j+k}=%
\begin{cases}
d^{\frac{n}{2}}W_{n+k}, & \text{if $n$ is even;} \\ 
d^{\frac{n-1}{2}}\left( W_{n+k+1}+qW_{n-k-1}\right) , & \text{if $n$ is odd.}%
\end{cases}%
\end{equation*}
\end{theorem}

\begin{proof}
By (\ref{equation3.2}), we have 
\begin{align}
( \mathbb{A}^2 -qI )^2 &= p^2 \mathbb{A}^2.  \notag
\end{align}%
Hence, we have the following identity: 
\begin{align}
(\mathbb{A}^2 +qI)^2 &=(\mathbb{A}^2 -qI)^2 + 4q \mathbb{A}^2 = (p^2 +4q) 
\mathbb{A}^2.  \label{equation3.4}
\end{align}
If $n$ is even, i.e., $n=2m$ where $m$ is a fixed non-negative integer, by (%
\ref{equation3.4}), then we have 
\begin{align}
\mathbb{A}^k (\mathbb{A}^2 +qI)^n =\mathbb{A}^k (\mathbb{A}^2 +qI
)^{2m}&=(p^2 +4q)^m \mathbb{A}^{2m+k}  \notag \\
\sum_{j=0}^n \binom{n}{j} q^{n-j} \mathbb{A}^{2j+k} &= (p^2 +4q)^m \mathbb{A}%
^{2m+k}.  \label{equation3.5}
\end{align}%
If $n$ is odd, i.e., $n=2m+1$ where $m$ is a fixed non-negative integer, by (%
\ref{equation3.4}), then we have 
\begin{align}
\mathbb{A}^k (\mathbb{A}^2 +qI )^{n}=\mathbb{A}^k (\mathbb{A}^2+qI
)^{2m+1}&=(p^2 +4q)^m \mathbb{A}^{2m+k} (\mathbb{A}^2+qI)  \notag \\
\sum_{j=0}^n \binom{n}{j} q^{n-j} \mathbb{A}^{2j+k} &= (p^2+4q)^m (\mathbb{A}%
^{2m+k+2} +q \mathbb{A}^{2m+k} ).  \label{equation3.6}
\end{align}%
We obtain the desired result by (\ref{equation3.1}) and comparing the second
entries in the matrices on both sides of the equation.
\end{proof}

For the $(p,q)$-Fibonacci sequence $\{u_n\}=\{w_n (0,1;p,q)\}$ and the $%
(p,q) $-Lucas sequence $\{v_n\}=\{w_n(2,p;p,q)\}$, the following identity
can be shown easily by induction on $n$: 
\begin{equation*}
u_{n+1} +q u_{n-1} =v_n.
\end{equation*}%
Based on this identity, it is clear that we have the same relationship
between the $(p,q)$-Fibonacci quaternions $U_n$ and the $(p,q)$-Lucas
quaternions $V_n$ as follows: 
\begin{equation*}
U_{n+1} + q U_{n-1} = V_n.
\end{equation*}%
Hence, we obtain Theorem 2.5 in Ipek's paper \cite{ipek} as a corollary of
Theorem \ref{theorem3.3}. For the sake of completeness of this paper, we
state it as follows:

\begin{corollary}
\label{corollary3.1} Let $n$ and $k$ be non-negative integers. Then, 
\begin{equation*}
\sum_{j=0}^{n}\binom{n}{j}q^{n-j}U_{2j+k}=%
\begin{cases}
d^{\frac{n}{2}}U_{n+k}, & \text{if $n$ is even;} \\ 
d^{\frac{n-1}{2}}V_{n+k}, & \text{if $n$ is odd.}%
\end{cases}%
\end{equation*}
\end{corollary}

We state the last binomial-sum identity for the Horadam quaternions as
follows:

\begin{theorem}
\label{theorem3.4} Let $n$ and $k$ be non-negative integers. Then, 
\begin{align*}
\sum_{j=0}^n \binom{n}{j} (-1)^j q^{n-j} W_{2j+k} &= (-p)^n W_{n+k}.
\end{align*}
\end{theorem}

\begin{proof}
By (\ref{equation3.2}), we have 
\begin{equation*}
( -\mathbb{A}^2 +qI )^n = (-p\mathbb{A})^n.
\end{equation*}%
Hence, we have 
\begin{equation*}
\mathbb{A}^k ( -\mathbb{A}^2 +qI )^n=\sum_{j=0}^n \binom{n}{j} (-1)^j
q^{n-j} \mathbb{A}^{2j+k} = (-p)^n \mathbb{A}^{n+k}.
\end{equation*}%
We obtained the desired result by (\ref{equation3.1}) and comparing the
second entries in the matrices on both sides of the equation.
\end{proof}

\begin{remark}
For the $(p,q)$-Fibonacci quaternions $U_{n}$, we obtain Theorem 2.6 in
Ipek's paper \cite{ipek} by Theorem \ref{theorem3.4}.
\end{remark}


\begin{thebibliography}{99}
\bibitem{alves} Alves FRV. The Quaterniontonic and Octoniontonic Fibonacci
Cassini's Identity: An Historical Investigation with the Maple's Help.
International Electronic Journal of Mathematics. 2018;13(3):125-138.

\bibitem{bitim} Bitim B. D. Some Identities of Fibonacci and Lucas
Quaternions by Quaternion Matrices. D\"{u}zce University Journal of Science
and Technology. 2019;7:606-615.

\bibitem{gould} Gould H.W. A history of the Fibonacci Q-matrix and a
higher-dimensional problem. Fibonacci Quart. 1981;19(3):250-257.

\bibitem{halici0} Halici S. On Fibonacci Quaternions. Adv. Appl. Clifford
Algebr. 2012;22:321-327.

\bibitem{halici} Halici S., Karata\c{s} A. On a Generalization for
Quaternion Sequences, Chaos, Solitons \& Fractals. 2017;98:178-182.

\bibitem{hamilton} Hamilton WR. Lectures on quaternions. Hodges and Smith.
Dublin; 1853.

\bibitem{horadam} Horadam AF. Basic Properties of a Certain Generalized
Sequence of Numbers. Fibonacci Quart. 1965;3(3):161-76.

\bibitem{horadam1} Horadam AF. Complex Fibonacci numbers and Fibonacci
quaternions. Am. Math. Mon. 1963;70:289-91.

\bibitem{ipek} Ipek A. On $(p,q)-$Fibonacci quaternions and their Binet
formulas, generating functions and certain binomial sums. Advances in
Applied Clifford Algebras. 2017: 27(2):1343-1351.

\bibitem{johnson} Johnson RC. Fibonacci Numbers and Matrices.
http://maths.dur.ac.uk/\symbol{126}dma0rcj/PED/fib.pdf. 2009; [Accessed 15
June 2009].

\bibitem{kesim} Kesim S, Polatli E. On quaternions with generalized
Fibonacci and Lucas number components. Advances in Difference Equations.
2015;169.

\bibitem{melham} Melham RS, Shannon AG. Some summation identities using
generalized Q-matrices. The Fibonaci Quarterly. 1995; 33(1):64-73.

\bibitem{patel} Patel BK, Ray PK. On the properties of $(p,q)-$Fibonacci and 
$(p,q)-$Lucas quaternions. Mathematical Reports. 2019; 21(71)1:15-25.

\bibitem{szynal1} Szynal-Liana A, Wloch I. The Pell Quaternions and the Pell
Octonions. Adv. Appl.Clifford Algebras. 2016;26:435-440.

\bibitem{szynal2} Szynal-Liana A, Wloch I. A note on Jacobsthal quaternions.
Adv. Appl. Clifford Algebras. 2016;26:441-7.

\bibitem{chaos1} Tan E, Yilmaz S, Sahin M. On a new generalization of
Fibonacci quaternions. Chaos, Solitons and Fractals. 2016;82:1-4.

\bibitem{tan} Sahin M, Tan E, Yilmaz S. The generalized bi-periodic
Fibonacci quaternions and octonions. Novi Sad J. Math. 2019;49(1):67-79.

\bibitem{t1} Tan E, Ekin AB. Some Identities On Conditional Sequences By
Using Matrix Method. Miskolc Mathematical Notes. 2017;18(1):469-477.

\bibitem{t2} Tan E. On bi-periodic Fibonacci and Lucas numbers by matrix
method. Ars Combinatoria. 2017;133:107-113.

\bibitem{morales} Cerda-Morales G. On a Generalization for Tribonacci
Quaternions. Mediterr. J. Math. 2017;14: 239.

\bibitem{morales1} Cerda-Morales G. Some Properties of Horadam quaternions,
arXiv:1707.05918.

\bibitem{thotsaporn} Tangboonduangjit A, Thanatipanonda T. Determinants
containing powers of generalized Fibonacci numbers. J. Integer Seq.
2016;19:16.7.1

\bibitem{t} Thanatipanonda TA. Fibonacci Identities through Matrix Method.
http://www.thotsaporn.com/FiboMatrix.pdf. 2018; [Accessed 25 December 2018].


\bibitem{waddill} Waddill M. Matrices and generalized Fibonacci sequences.
Fibonacci Quart. 1974;12:381-386.
\end{thebibliography}
\end{document}